%% file: main.tex
\documentclass[a4paper]{article}
\input format.tex
\begin{document}%
\title{Strongly light subgraphs in the $1$-planar graphs with minimum degree~$7$}
\author{Tao Wang\,\textsuperscript{a, b, }\footnote{{\tt Corresponding
author: wangtao@henu.edu.cn} }\\
{\small \textsuperscript{a}Institute of Applied Mathematics}\\
{\small Henan University, Kaifeng, 475004, P. R. China}\\
{\small \textsuperscript{b}School of Mathematics and Statistics}\\
{\small Henan University, Kaifeng, 475004, P. R. China}}
\date{May 30, 2015}
\maketitle
\begin{abstract}%
A graph is {\em $1$-planar} if it can be drawn in the plane such that every edge crosses at most one other edge. A connected graph $H$ is {\em strongly light} in a family of graphs $\mathfrak{G}$, if there exists a constant $\lambda$, such that every graph $G$ in $\mathfrak{G}$ contains a subgraph $K$ isomorphic to $H$ with $\deg_{G}(v) \leq \lambda$ for all $v \in V(K)$. In this paper, we present some strongly light subgraphs in the family of $1$-planar graphs with minimum degree~$7$.
\end{abstract}
\section{Introduction}
All graphs considered are finite, simple and undirected unless otherwise stated. We denote by $V(G)$ and $E(G)$ the vertex set and the edge set of $G$. We shall denote by $F(G)$ the set of faces of an embedded graph $G$. The {\em degree} of a vertex $v$ in $G$, denoted by $\deg_{G}(v)$, is the number of edges of $G$ incident with $v$. We denote the minimum and maximum degrees of vertices in $G$ by $\delta(G)$ and $\Delta(G)$, respectively. A {\em wheel} $W_{n}$ is a graph obtained by taking the join of a cycle $C_{n}$ and a single vertex. In an embedded graph $G$, the {\em degree} of a face $f$, denoted by $\deg_{G}(f)$, is the number of edges with which it is incident, each cut edge being counted twice. A $k$-vertex, $k^{+}$-vertex and $k^{-}$-vertex is a vertex of degree $k$, at least $k$ and at most $k$, respectively. Similarly, we can define $k$-face, $k^{+}$-face and $k^{-}$-face. 

A graph is {\em $1$-embeddable} in a surface $S$ if it can be drawn in $S$ such that every edge crosses at most one other edge. In particular, a graph is {\em $1$-planar} if it can be drawn in the plane such that every edge crosses at most one other edge. The concept of $1$-planar graph was introduced by Ringel \cite{MR0187232} in 1965, while he simultaneously colors the vertices and faces of a plane graph such that any pair of adjacent/incident elements receive different colors. Ringel \cite{MR0187232} proved that every $1$-planar graph is $7$-colorable, and conjectured that every $1$-planar graph is $6$-colorable. In 1984, Borodin \cite{MR832128} confirmed this conjecture, and later Borodin \cite{MR1333779} found a better proof for it. Recently, various coloring problems of $1$-planar graphs are considered, see \cite{MR1865580, MR2779909, MR2876230}. 

A connected graph $H$ is {\em strongly light} in a family of graphs $\mathfrak{G}$, if there exists an integer $\lambda$, such that every graph $G$ in $\mathfrak{G}$ contains a subgraph $K$ isomorphic to $H$ with $\deg_{G}(v) \leq \lambda$ for all $v \in V(K)$. A graph $H$ is said to be {\em light} in a family $\mathfrak{G}$ of graphs if at least one member of $\mathfrak{G}$ contains a copy of $H$ and there is an integer $\lambda(H, \mathfrak{G})$ such that each member $G$ of $\mathfrak{G}$ with a copy of $H$ also has a copy $K$ of $H$ such that $\deg_{G}(v) \leq \lambda(H, \mathfrak{G})$ for all $v \in V(K)$. Note that a light subgraph may be not strongly light, for example, the graph $K_{5}$ is light in the family of graphs $\mathfrak{G} = \{\textrm{planar graphs}\} \cup \{K_{6}\}$, but $K_{5}$ is not strongly light in $\mathfrak{G}$ since not every graph in $\mathfrak{G}$ contains a subgraph $K_{5}$. The light subgraphs  are well studied when $\mathfrak{G}$ is a subclass of planar graphs, and we refer the reader to a good survey \cite{MR3004475}.

Fabrici and Madaras \cite{MR2297168} studied the structure of $1$-planar graphs, mainly on the light subgraphs of $1$-planar graphs. They showed that every $3$-connected $1$-planar graph contains an edge with each end having degree at most $20$, and this bound is the best possible. Hud\'{a}k and Madaras \cite{MR2574477} proved that each $1$-planar graph of minimum degree $5$ and girth $4$ contains (i) a 5-vertex adjacent to a vertex of degree at most 6, (ii) a $4$-cycle whose vertices all have degree at most $9$ (the upper bound was further improved to $8$ by Borodin, Dmitriev and Ivanova \cite{MR2543605}), (iii) a star $K_{1,4}$ with all vertices having degree at most $11$. 

In 1965, Ringel \cite{MR0187232} found that each $1$-planar graph has a vertex of degree at most $7$ and the bound is tight. Hud\'{a}k and Madaras \cite{MR2802062} considered strongly light subgraphs in the family of $1$-planar graphs with minimum degree~$7$, and proved the following theorem.

\begin{theorem}[Hud\'{a}k and Madaras \cite{MR2802062}]
Each $1$-planar graph with minimum degree~$7$ contains
\begin{enumerate}[label=(\alph*)]
\item two adjacent $7$-vertices,
\item a $K_{4}$ whose vertices all have degree at most $13$,
\item a $K_{2, 3}^{*}$ whose vertices all have degree at most $13$, where $K_{2, 3}^{*}$ is a graph $K_{2, 3}$ with an extra edge between two vertices of the smaller bipartition,
\item a $W_{4}$ whose vertices all have degree at most $11$.
\end{enumerate}
\end{theorem}

In this paper, we also consider strongly light subgraphs in the family of $1$-planar graphs with minimum degree~$7$. 

\section{Strongly light subgraphs}
Let $G$ be a graph having been drawn in a surface; if we treat all the crossing points as vertices, then we obtain an embedded graph $G^{\dagger}$, and call it {\em the associated graph of $G$}, call the vertices of $G$ {\em true vertices} and the crossing points {\em crossing vertices}. In the associated graph, a $3$-face is called a {\em false $3$-face} if it is incident with a crossing vertex; otherwise, it is a {\em true $3$-face}. Clearly, a false $3$-face is incident with exactly one crossing vertex. Note that the set of crossing vertices in the associated graph is independent. In the figures of this paper, the solid dots denote true vertices and the hollow dots denote crossing vertices, and some degree restrictions are beside the vertices. 

Zhang \etal presented two strongly light subgraphs on four vertices in the family of $1$-planar graphs with minimum degree~$7$. 
\begin{theorem}[Zhang \etal \cite{MR2832629}]%
Each $1$-planar graph with minimum degree~$7$ contains a $K_{4}$ with all vertices of degree at most~$11$.
\end{theorem}

\begin{theorem}[Zhang \etal \cite{MR3240437}]\label{Chorld4Cycle}%
Each $1$-planar graph with minimum degree~$7$ contains a $4$-cycle $C = [x_{1}x_{2}x_{3}x_{4}]$ with a chord $x_{1}x_{3}$, where $\deg(x_{1}) = 7$, $\deg(x_{2}) \leq 10$, $\deg(x_{3}) \leq 8$ and $\deg(x_{4}) \leq 10$.
\end{theorem}

We improve the above two results to the following. A $K_{4}$ is of type $(d_{1}, d_{2}, d_{3}, d_{4})$ if its degrees are $d_{1}, d_{2}, d_{3}$ and $d_{4}$, respectively. Similarly, we can define a $K_{4}$ of type $(d_{1}^{+}, d_{2}^{+}, d_{3}^{+}, d_{4}^{+})$, etc. 
\begin{theorem}%
If $G$ is a $1$-planar graph with minimum degree~$7$, then it contains a $K_{4}$ of the type $(7, 8^{-}, 8^{-}, 10^{-})$.
\end{theorem}
\begin{proof}%
Suppose that $G$ is a connected counterexample to the theorem, which implies that $G$ contains no $K_{4}$ or every copy of $K_{4}$ is of the type $(8^{+}, 8^{+}, 8^{+}, 8^{+})$ or $(7, 9^{+}, 9^{+}, 9^{+})$ or $(7, 8^{-}, 9^{+}, 9^{+})$ or $(7, 8^{-}, 8^{-}, 11^{+})$.

Furthermore, we may assume that $G$ has been $1$-embedded in the plane. Clearly, every face of its associated graph is homeomorphic to an open disk.  Let $K^{\dagger}$ be the associated graph of $G$. By Euler's formula, we have
\begin{equation}\label{EQ1}%
\sum_{v \in V(K^{\dagger})}(\deg_{K^{\dagger}}(v)-6) + \sum_{f \in F(K^{\dagger})}(2\deg_{K^{\dagger}}(f) - 6) = -12.
\end{equation}

We will use the discharging method to complete the proof. The initial charge of every vertex $v$ is $\deg_{K^{\dagger}}(v)-6$, and the initial charge of every face $f$ is $2\deg_{K^{\dagger}}(f) - 6$. By \eqref{EQ1}, the sum of all the elements' charge is $-12$. We then transfer some charge from the $4^{+}$-faces and the $7^{+}$-vertices to crossing vertices, such that the final charge of every crossing vertex becomes nonnegative and the final charge of every $4^{+}$-face and every $7^{+}$-vertex remains nonnegative, and thus the sum of the final charge of vertices and faces is nonnegative, which leads to a contradiction.
\paragraph{The Discharging Rules:}
\begin{enumerate}[label=(R\arabic*)]
\item every $4^{+}$-face donates its redundant charge equally to incident crossing vertices;
\item every $7^{+}$-vertex donates its redundant charge equally to incident false $3$-faces;
\item after applying (R2), every false $3$-face donates its redundant charge to the incident crossing vertex.
\end{enumerate}

By the discharging rules, the final charge of every face and every $7^{+}$-vertex is nonnegative. So it suffices to consider the final charge of crossing vertices in $K^{\dagger}$.

By the construction of $K^{\dagger}$, every face is incident with at most $\left\lfloor \frac{\deg_{K^{\dagger}}(v)}{2} \right\rfloor$ crossing vertices. Thus, every $4^{+}$-face sends at least $1$ to each incident crossing vertex.

Note that every $7^{+}$-vertex $v$ is incident with at most $2\left\lfloor \frac{\deg_{K^{\dagger}}(v)}{2} \right\rfloor$ false $3$-faces. More formally, every $7$-vertex sends at least $\frac{1}{6}$ to each incident false $3$-face; every $8$-vertex sends at least $\frac{1}{4}$ to each incident false $3$-face; every $9$-vertex sends at least $\frac{3}{8}$ to each incident false $3$-face; every $10$-vertex sends at least $\frac{2}{5}$ to each incident false $3$-face; every $11^{+}$-vertex sends at least $\frac{1}{2}$ to each incident false $3$-face. 

Let $w$ be an arbitrary crossing vertex in $K^{\dagger}$. Notice that the four neighbors of $w$ are $7^{+}$-vertices.

If $w$ is incident with at least two $4^{+}$-faces, then its final charge is greater than $4 - 6 + 2 \times 1 = 0$. If $w$ is incident with exactly one $4^{+}$-face, then its final charge is at least $4 - 6 + 1 + 6 \times \frac{1}{6} = 0$.

If there is no crossing vertex which is incident with four $3$-faces, then the sum of the final charge is nonnegative, which leads to a contradiction. So we may assume that $w$ is incident with four $3$-faces. It is obvious that the four neighbors of $w$ induce a $K_{4}$ in $G$. If this $K_{4}$ is of the type $(8^{+}, 8^{+}, 8^{+}, 8^{+})$, then the final charge of $w$ is at least $4 - 6 + 8 \times \frac{1}{4} = 0$; if this $K_{4}$ is of the type $(7, 9^{+}, 9^{+}, 9^{+})$, then the final charge of $w$ is at least $4 - 6 + 2 \times \frac{1}{6} + 6 \times \frac{3}{8} > 0$; if this $K_{4}$ is of the type $(7, 8^{-}, 9^{+}, 9^{+})$, then the final charge of $w$ is at least $4 - 6 + 4 \times \frac{1}{6} + 4 \times \frac{3}{8} > 0$; if this $K_{4}$ is of the type $(7, 8^{-}, 8^{-}, 11^{+})$, then the final charge of $w$ is at least $4 - 6 + 6 \times \frac{1}{6} + 2 \times \frac{1}{2} = 0$.

Finally, all the faces and vertices have nonnegative charge, which leads to a contradiction.
\end{proof}

To the author's knowledge, all the known strongly light graphs have at most five vertices. Now, we give a strongly light graph on 8 vertices in the family of $1$-planar graphs with minimum degree~$7$.
\begin{theorem}\label{W4+}
If $G$ is a $1$-planar graph with minimum degree~$7$, then $G$ contains a subgraph as illustrated in Fig.~\subref{fig:subfig:star}. Moreover, 
(i) every vertex in $\{w_{2}, w_{3}, \dots, w_{7}\}$ has degree at most $23$; (ii) at most one vertex in $\{w_{2}, w_{3}, \dots, w_{7}\}$ is a $12^{+}$-vertex; (iii) if no vertex in $\{w_{2}, w_{3}, w_{5}, w_{7}\}$ is a $7$-vertex, then $w_{2}w_{3}, w_{3}w_{4}, w_{4}w_{5}, w_{5}w_{6}, w_{6}w_{7}, w_{7}w_{1} \in E(G)$. 
\end{theorem}
\begin{figure}%
\centering
\subcaptionbox{\label{fig:subfig:star}}{\includegraphics{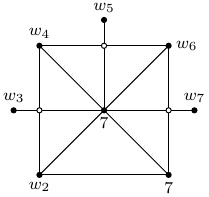}}
\end{figure}
\begin{proof}%
Suppose that $G$ is a connected $1$-planar graph with minimum degree~$7$, and it has been $1$-embedded in the plane. Clearly, every face of its associated graph is homeomorphic to an open disk. Let $K^{\dagger}$ be the associated graph of $G$.

By Euler's formula, we have
\begin{equation}\label{EQ2}%
\sum_{v \in V(K^{\dagger})}(\deg_{K^{\dagger}}(v) - 4) + \sum_{f \in F(K^{\dagger})}(\deg_{K^{\dagger}}(f) - 4) = -8.
\end{equation}

We will use the discharging method to complete the proof. The initial charge of every vertex $v$ is $\deg_{K^{\dagger}}(v) - 4$, and the initial charge of every face $f$ is $\deg_{K^{\dagger}}(f) - 4$. By \eqref{EQ2}, the sum of all the elements' charge is $-8$. We then transfer some charge from the $7^{+}$-vertices to the $3$-faces, such that the final charge of every face and every $8^{+}$-vertex is nonnegative, thus there exists a $7$-vertex such that its final charge is negative and the local structure is desired.

\paragraph{The Discharging Rules:}
\begin{enumerate}[label=(R\arabic*)]
\item every $7^{+}$-vertex sends $\frac{1}{2}$ to each incident false $3$-face and sends $\frac{1}{3}$ to each incident true $3$-face;
\item let $f$ be a face with a face angle $w_{1}ww_{2}$ and $\deg(w) = k \geq 8$,
\begin{enumerate}
\item if $f$ is a $3$-face with $\deg(w_{1}) = 7$ and $\deg(w_{2}) \geq 8$, then $w$ sends $\frac{k-4}{k} - \frac{1}{3}$ to $w_{1}$ through $f$;
\item if $f$ is a $3$-face with $\deg(w_{1}) = \deg(w_{2}) = 7$, then each of $w_{1}$ and $w_{2}$ receives $\frac{k-4}{2k} - \frac{1}{6}$ from $w$ through $f$;
\item if $f$ is a false $3$-face with crossing vertex $w_{1}$ and $w_{1}$ is on the edge $uw$ of $G$, then $w$ sends $\frac{k-4}{2k} - \frac{1}{4}$ to $w_{2}$ through $f$, and additionally $w$ sends $\frac{k-4}{2k} - \frac{1}{4}$ to $u$ through $f$;
\item if $f$ is a $4^{+}$-face with crossing vertex $w_{1}$ and $w_{1}$ is on the edge $uw$ of $G$, then $w$ sends $\frac{k-4}{2k}$ to $u$ through $f$.
\end{enumerate}
\end{enumerate}
By the discharging rules, the final charge of every face and every $8^{+}$-vertex is nonnegative. Hence, there exists a $7$-vertex $w_{0}$ such that its final charge is negative.

If $w_{0}$ is incident with at least one $4^{+}$-face, then its final charge is at least $7 - 4 - 6 \times \frac{1}{2} = 0$. So we may assume that $w_{0}$ is incident with seven $3$-faces. Notice that the number of incident false $3$-faces is even. If $w_{0}$ is incident with at most four false $3$-faces, then its final charge is at least $7 - 4 - 4 \times \frac{1}{2} - 3 \times \frac{1}{3} = 0$. Hence, the vertex $w_{0}$ must be incident with six false $3$-faces and one true $3$-face. We also notice that $w_{0}$ receives less than $\frac{1}{3}$ from all the other vertices; otherwise, its final charge is at least $7 - 4 + \frac{1}{3} - 6 \times \frac{1}{2} - \frac{1}{3} = 0$.

Let $w_{1}w_{0}w_{2}$ be the true $3$-face. If both $w_{1}$ and $w_{2}$ are $8^{+}$-vertices, then $w_{0}$ receives at least $\frac{1}{2} - \frac{1}{3} = \frac{1}{6}$ from each of $w_{1}$ and $w_{2}$ by (R2-a), thus $w_{0}$ receives at least $\frac{1}{3}$ from all the other vertices, a contradiction. Hence, at least one of $w_{1}$ and $w_{2}$ must be a $7$-vertex, so we may assume that $w_{1}$ is a $7$-vertex, see Fig.~\subref{fig:subfig:star}.

(i) Suppose that $w_{0}$ is adjacent to a $24^{+}$-vertex $w$ in $G$. By the discharging rules, the vertex $w_{0}$ receives at least $2 \times (\frac{5}{12} - \frac{1}{4}) = \frac{1}{3}$ from $w$, which leads to a contradiction. Hence, every vertex in $\{w_{2}, w_{3}, \dots, w_{7}\}$ has degree at most~$23$.

(ii) If at least two vertices in $\{w_{2}, w_{3}, \dots, w_{7}\}$ are $12^{+}$-vertices, then $w_{0}$ will receive at least $4 \times (\frac{1}{3} - \frac{1}{4}) = \frac{1}{3}$, which leads to a contradiction. Hence, at most one vertex in $\{w_{2}, w_{3}, \dots, w_{7}\}$ is a $12^{+}$-vertex.

(iii) Suppose, to derive a contradiction, that $w_{2}w_{3}, w_{3}w_{4}, w_{4}w_{5}, w_{5}w_{6}, w_{6}w_{7}, w_{7}w_{1} \in E(G)$ does not hold. Thus, at least one crossing vertex in Fig.~\subref{fig:subfig:star} is incident with a $4^{+}$-face. By (R2-d), the vertex $w_{0}$ receives at least $\frac{1}{4}$ from a $8^{+}$-vertex through a $4^{+}$-face. By (R2-b), the vertex $w_{0}$ receives at least $\frac{1}{4} - \frac{1}{6} = \frac{1}{12}$ from $w_{2}$ through the true $3$-face $w_{0}w_{1}w_{2}$, thus it receives at least $\frac{1}{4} + \frac{1}{12} = \frac{1}{3}$ from all the other vertices, which derives a contradiction.   
\end{proof}
\begin{corollary}[Hud\'{a}k and Madaras \cite{MR2802062}]
If $G$ is a $1$-planar graph with minimum degree~$7$, then it contains an edge such that each end has degree exactly $7$. 
\end{corollary}
\begin{corollary}%
Every $1$-planar graph with minimum degree~$7$ contains a $K_{1, 7}$ with the center of degree $7$ and the other vertices of degree at most $23$.
\end{corollary}

By \autoref{W4+}, the wheel $W_{4}$ is strongly light in the family of $1$-planar graphs with minimum degree~$7$. In the next theorem, we further improve the degree restriction on each vertex in $W_{4}$. 

\begin{figure}%
\ContinuedFloat
\centering
\subcaptionbox{\label{fig:subfig:W_41}}{\includegraphics[width = 0.15\textwidth]{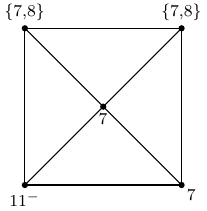}}\hfill~
\subcaptionbox{\label{fig:subfig:W_42}}{\includegraphics[width = 0.15\textwidth]{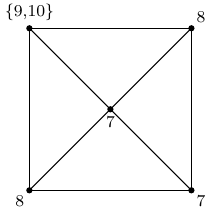}}\hfill~
\subcaptionbox{\label{fig:subfig:W_43}}{\includegraphics[width = 0.15\textwidth]{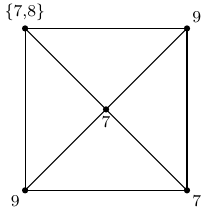}}\hfill~
\subcaptionbox{\label{fig:subfig:W_44}}{\includegraphics[width = 0.15\textwidth]{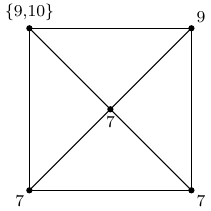}}\hfill~
\subcaptionbox{\label{fig:subfig:W_45}}{\includegraphics[width = 0.15\textwidth]{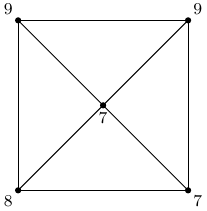}}
\end{figure}

\begin{theorem}\label{LW4}
If $G$ is a 1-planar graph with minimum degree 7, then $G$ contains at least one subgraph as illustrated in Fig.~\subref{fig:subfig:W_41}--\subref{fig:subfig:W_45}.
\end{theorem}
\begin{proof}%
Suppose that $G$ is a connected $1$-planar graph with minimum degree~$7$, and it has been $1$-embedded in the plane. Clearly, every face of its associated graph is homeomorphic to an open disk. Let $K^{\dagger}$ be the associated graph of $G$.

By Euler's formula, we have
\begin{equation}\label{EQ3}%
\sum_{v \in V(K^{\dagger})}(\deg_{K^{\dagger}}(v) - 4) + \sum_{f \in F(K^{\dagger})}(\deg_{K^{\dagger}}(f) - 4) = -8.
\end{equation}

We will use the discharging method to complete the proof. The initial charge of every vertex $v$ is $\deg_{K^{\dagger}}(v) - 4$, and the initial charge of every face $f$ is $\deg_{K^{\dagger}}(f) - 4$. By \eqref{EQ3}, the sum of all the elements' charge is $-8$. We then transfer some charge from the $7^{+}$-vertices to the $3$-faces, such that the final charge of every face and every $8^{+}$-vertex is nonnegative, thus there exists a $7$-vertex such that its final charge is negative and the local structure is desired.

\paragraph{The Discharging Rules:}
\begin{enumerate}[label=(R\arabic*)]
\item every $7^{+}$-vertex sends $\frac{1}{2}$ to each incident false $3$-face and sends $\frac{1}{3}$ to each incident true $3$-face;
\item let $f$ be a face with a face angle $w_{1}ww_{2}$ and $\deg(w) = k \geq 8$,
\begin{enumerate}
\item if $f$ is a $3$-face with $\deg(w_{1}) = 7$ and $\deg(w_{2}) \geq 8$, then $w$ sends $\frac{k-4}{k} - \frac{1}{3}$ to $w_{1}$ through $f$;
\item if $f$ is a $3$-face with $\deg(w_{1}) = \deg(w_{2}) = 7$, then each of $w_{1}$ and $w_{2}$ receives $\frac{k-4}{2k} - \frac{1}{6}$ from $w$ through $f$;
\item if $f$ is a false $3$-face with crossing vertex $w_{1}$, then $w$ sends $\frac{k-4}{k} - \frac{1}{2}$ to $w_{2}$ through $f$.
\end{enumerate}
\end{enumerate}
By the discharging rules, the final charge of every face and every $8^{+}$-vertex is nonnegative. Hence, there exists a $7$-vertex $w_{0}$ such that its final charge is negative.

If $w_{0}$ is incident with at least one $4^{+}$-face, then its final charge is at least $7 - 4 - 6 \times \frac{1}{2} = 0$. So we may assume that $w_{0}$ is incident with seven $3$-faces. Notice that the number of incident false $3$-faces is even. If $w_{0}$ is incident with at most four false $3$-faces, then its final charge is at least $7 - 4 - 4 \times \frac{1}{2} - 3 \times \frac{1}{3} = 0$. Hence, the vertex $w_{0}$ must be incident with six false $3$-faces and one true $3$-face. We also notice that $w_{0}$ receives less than $\frac{1}{3}$ from all the other vertices; otherwise, its final charge is at least $7 - 4 + \frac{1}{3} - 6 \times \frac{1}{2} - \frac{1}{3} = 0$.

Let $w_{1}w_{0}w_{2}$ be the true $3$-face. If both $w_{1}$ and $w_{2}$ are $8^{+}$-vertices, then $w_{0}$ receives at least $\frac{1}{2} - \frac{1}{3} = \frac{1}{6}$ from each of $w_{1}$ and $w_{2}$ by (R2-a), thus $w_{0}$ receives at least $\frac{1}{3}$ from all the other vertices, a contradiction. Hence, at least one of $w_{1}$ and $w_{2}$ must be a $7$-vertex, so we may assume that $w_{1}$ is a $7$-vertex, see Fig.~\subref{fig:subfig:star}.

\begin{case}
Both $\deg(w_{4})$ and $\deg(w_{6})$ belong to $\{7, 8\}$. 
\end{case}
Since the vertex $w_{0}$ receives less than $\frac{1}{3}$ from the vertex $w_{2}$, it follows that $(\frac{\deg(w_{2}) - 4}{\deg(w_{2})} - \frac{1}{2}) + (\frac{\deg(w_{2}) - 4}{2\deg(w_{2})}- \frac{1}{6}) < \frac{1}{3}$ and $\deg(w_{2}) < 12$, see Fig.~\subref{fig:subfig:W_41}.

\begin{case}
Exactly one of $\deg(w_{4})$ and $\deg(w_{6})$ belongs to $\{7, 8\}$.
\end{case}
Note that $\max\{\deg(w_{4}), \deg(w_{6})\} \geq 9$, if $w_{2}$ is a $10^{+}$-vertex, then $w_{0}$ receives at least $2 \times (\frac{5}{9} - \frac{1}{2}) + (\frac{3}{5} - \frac{1}{2}) + (\frac{3}{10} - \frac{1}{6}) > \frac{1}{3}$, a contradiction. So we may assume that $w_{2}$ is a $9^{-}$-vertex. If $w_{2}$ is a $7$-vertex and $\max\{\deg(w_{4}), \deg(w_{6})\} \geq 12$, then $w_{0}$ will receive at least $2 \times (\frac{2}{3} - \frac{1}{2}) = \frac{1}{3}$, which is a contradiction.  If $w_{2}$ is a $8$-vertex and $\max\{\deg(w_{4}), \deg(w_{6})\} \geq 11$, then $w_{0}$ will receive at least $2 \times (\frac{7}{11} - \frac{1}{2}) + (\frac{1}{4} - \frac{1}{6}) > \frac{1}{3}$, a contradiction. If $w_{2}$ is a $9$-vertex and $\max\{\deg(w_{4}), \deg(w_{6})\} \geq 10$, then $w_{0}$ receives at least $2 \times (\frac{3}{5} - \frac{1}{2}) + (\frac{5}{9} - \frac{1}{2}) + (\frac{5}{18} - \frac{1}{6}) = \frac{11}{30} > \frac{1}{3}$, which leads to a contradiction. In summary, if $w_{2}$ is a $7$-vertex, then $\max\{\deg(w_{4}), \deg(w_{6})\} \in \{9, 10, 11\}$, and thus $G$ contains a subgraph isomorphic to that in Fig.~\subref{fig:subfig:W_41}; if $w_{2}$ is a $8$-vertex, then $\max\{\deg(w_{4}), \deg(w_{6})\} \in \{9, 10\}$, and thus $G$ contains a subgraph isomorphic to that in Fig.~\subref{fig:subfig:W_41} or Fig.~\subref{fig:subfig:W_42}; if $w_{2}$ is a $9$-vertex, then $\max\{\deg(w_{4}), \deg(w_{6})\} = 9$, and thus $G$ contains a subgraph isomorphic to that in Fig.~\subref{fig:subfig:W_43}, Fig.~\subref{fig:subfig:W_44} or Fig.~\subref{fig:subfig:W_45}.

\begin{case}
Both $\deg(w_{4})$ and $\deg(w_{6})$ are at least $9$. 
\end{case}
If $w_{2}$ is a $9^{+}$-vertex, then the vertex $w_{0}$ will receive at least $(\frac{5}{18} -\frac{1}{6}) +  5 \times (\frac{5}{9} - \frac{1}{2}) > \frac{1}{3}$, a contradiction. So we may assume that $w_{2}$ is a $7$- or $8$-vertex. If $\min\{\deg(w_{4}), \deg(w_{6})\} \geq 10$, then the vertex $w_{0}$ will receive at least $4 \times (\frac{3}{5} - \frac{1}{2}) = \frac{2}{5} > \frac{1}{3}$, a contradiction. Hence, we have that $\min\{\deg(w_{4}), \deg(w_{6})\} = 9$. If $w_{2}$ is a $7$-vertex and $\max\{\deg(w_{4}), \deg(w_{6})\} \geq 11$, then the vertex $w_{0}$ will receive at least $2 \times (\frac{7}{11} - \frac{1}{2}) + 2 \times (\frac{5}{9} - \frac{1}{2}) > \frac{1}{3}$, a contradiction. If $w_{2}$ is a $8$-vertex and $\max\{\deg(w_{4}), \deg(w_{6})\} \geq 10$, then $w_{0}$ will receive at least $2 \times (\frac{3}{5} - \frac{1}{2}) + 2 \times (\frac{5}{9} - \frac{1}{2}) + (\frac{1}{4} - \frac{1}{6}) > \frac{1}{3}$, which is a contradiction. In summary, if $w_{2}$ is a $7$-vertex, then $G$ contains a subgraph as illustrated in Fig.~\subref{fig:subfig:W_44}; if $w_{2}$ is a $8$-vertex, then $G$ contains a subgraph as illustrated in Fig.~\subref{fig:subfig:W_45}. 
\end{proof}

\begin{corollary}%
If $G$ is a 1-planar graph with minimum degree 7, then $G$ contains a triangle having vertex degree $7, 7$ and at most $9$, respectively. 
\end{corollary}
As an immediate consequence of \autoref{LW4}, the following corollary is an improvement of  \autoref{Chorld4Cycle}.
\begin{corollary}%
If $G$ is a 1-planar graph with minimum degree 7, then $G$ contains a $4$-cycle $C = [x_{1}x_{2}x_{3}x_{4}]$ with a chord $x_{1}x_{3}$, where $\deg(x_{1}) = 7$, $\deg(x_{2}) \leq 9$, $\deg(x_{3}) \leq 8$ and $\deg(x_{4}) \leq 9$.
\end{corollary}

\begin{corollary}[\cite{MR2916329}]%
If $G$ is a 1-planar graph with minimum degree 7, then $G$ contains a copy of $K_{1} \vee (K_{1} \cup K_{2})$ with all the vertices of degree at most $9$.
\end{corollary}

\vskip 0mm \vspace{0.3cm} \noindent{\bf Acknowledgments.} The author was supported by NSFC (11101125) and partially supported by the Fundamental Research Funds for Universities in Henan. The author would like to thank the anonymous reviewers for their valuable comments and assistance on earlier drafts.

\end{document}

%% file: format.tex
\usepackage{mathrsfs, amsmath, amsthm}
\usepackage{graphicx, color}
\usepackage{caption, subcaption}
\usepackage{array}
\usepackage{cases}
\makeatletter
\def\th@plain{%
  \upshape 
}
\makeatother

\makeatletter
\renewenvironment{proof}[1][\proofname]{\par
  \pushQED{\qed}%
  \normalfont \topsep6\p@\@plus6\p@\relax
  \trivlist
  \item[\hskip\labelsep
        \bfseries
    #1\@addpunct{.}]\ignorespaces
}{%
  \popQED\endtrivlist\@endpefalse
}
\makeatother

\newtheorem{theorem}{Theorem}
\numberwithin{theorem}{section}

\newtheorem{corollary}{Corollary}

\newtheorem*{conjecture*}{Conjecture}
\newtheorem{case}{Case}

\theoremstyle{definition}

\usepackage[left=25mm,top=25mm,bottom=25mm,right=25mm]{geometry}
\usepackage[T1]{fontenc}
\usepackage[varg]{txfonts}

\usepackage{enumitem}
\usepackage[square, numbers, sort&compress]{natbib}

\ifx\pdfoutput\undefined
 \usepackage[dvipdfm,%
  bookmarks=true,%
  bookmarksnumbered=true, 
  bookmarksopen=true, 
  plainpages=false,%
  pdfpagelabels,%
  colorlinks=true, 
  linkcolor=blue, 
  citecolor=blue,%
  hyperindex=true,
  urlcolor=black]{hyperref}
\else
 \usepackage[pdftex,%
  bookmarks=true,%
  bookmarksnumbered=true, 
  bookmarksopen=true, 
  plainpages=false,%
  pdfpagelabels,%
  colorlinks=true, 
  linkcolor=blue, 
  citecolor=blue,%
  anchorcolor=green,
  urlcolor= blue,
  breaklinks=true,
  hyperindex=true]{hyperref}
\fi

\renewcommand{\figurename}{Fig.}

\newcounter{Hcase}
\newcounter{Hclaim}

\newcommand{\etal}{et~al.\ }

\def\int(#1){\mathrm{int}(#1)}
\def\ext(#1){\mathrm{ext}(#1)}
\def\Int(#1){\mathrm{Int}(#1)}
\def\Ext(#1){\mathrm{Ext}(#1)}
\def\ad(#1){\mathrm{ad}(#1)}
\def\mad(#1){\mathrm{mad}(#1)}
\def\la(#1){\mathrm{la}(#1)}

